\documentclass[a4paper, 11pt]{amsart}
\usepackage[english]{babel}
\usepackage[utf8]{inputenc}
\usepackage{amsthm, amsmath, amssymb, amsrefs, enumerate, enumitem, mathtools, nicefrac, bbm}

\usepackage[top=1.05in, bottom=1.05in, left=0.90in, right=0.90in]{geometry}
\binoppenalty=\maxdimen
\relpenalty=\maxdimen

\usepackage[T1]{fontenc}
\usepackage{lmodern}

\usepackage[colorlinks=true,linkcolor=black,citecolor=black,filecolor=black]{hyperref}

\theoremstyle{plain}
\newtheorem{thm}{Theorem}[section]
\newtheorem{prop}[thm]{Proposition}
\newtheorem{lemma}[thm]{Lemma}
\newtheorem{cor}[thm]{Corollary}
\newtheorem*{que}{Question}
\theoremstyle{definition}
\newtheorem{defn}[thm]{Definition}

\theoremstyle{remark}
\newtheorem*{rmk}{Remark}
\newtheorem*{rmks}{Remarks}

\usepackage{xpatch}
\makeatletter
\AtBeginDocument{\xpatchcmd{\@thm}{\thm@headpunct{.}}{\thm@headpunct{}}{}{}}
\xpatchcmd{\proof}{\@addpunct{.}}{\@addpunct{:}}{}{}
\makeatother

\makeatletter
\let\@@pmod\pmod
\DeclareRobustCommand{\pmod}{\@ifstar\@pmods\@@pmod}
\def\@pmods#1{\mkern4mu({\operator@font mod}\mkern 6mu#1)}
\makeatother

\def\C{\mathbb{C}}
\def\H{\mathbb{H}}
\newcommand{\Z}{\mathbb{Z}}
\newcommand{\Q}{\mathbb{Q}}
\newcommand{\N}{\mathbb{N}}
\newcommand{\R}{\mathbb{R}}
\newcommand{\tp}{\theta_{\psi}}
\newcommand{\tc}{\theta_{\chi}}
\newcommand{\ti}{\theta_{\mathbbm{1}}}
\newcommand{\Mp}{M_{\psi}}
\def\Mc{M_{\chi}}
\newcommand{\lcm}{\textnormal{lcm}}
\newcommand{\sol}{\sigma^{\text{sm}}_{1,\psi}}
\newcommand{\sdf}{\sigma^{\text{sm}}_{2,\chi}}
\newcommand{\sdft}{\sigma^{\text{sm}}_{2,\mathbbm{1}}}
\newcommand{\Ec}{\mathcal{E}}
\newcommand{\Fc}{\mathcal{F}}
\newcommand{\Gc}{\mathcal{G}}
\newcommand{\Fcp}{\mathcal{F}^{+}}
\newcommand{\Fcm}{\mathcal{F}^{-}}
\newcommand{\Gcp}{\mathcal{G}^{+}}
\newcommand{\Gcm}{\mathcal{G}^{-}}
\newcommand{\Hcp}{\mathcal{H}^{+}}
\newcommand{\Pc}{\mathcal{P}}
\newcommand{\slz}{{\text {\rm SL}}_2(\mathbb{Z})}
\newcommand{\re}{\textnormal{Re}}

\DeclarePairedDelimiter\floor{\lfloor}{\rfloor}

\title{Polar harmonic Maa{\SS} forms and holomorphic projection}
\author{Joshua Males}
\address{450 Machray Hall, Department of Mathematics, University of Manitoba, Winnipeg,
	Canada}
\email{joshua.males@umanitoba.ca}
\author{Andreas Mono}
\address{Department of Mathematics and Computer Science, Division of Mathematics, University of Cologne, Weyertal 86-90, 50931 Cologne, Germany}
\email{amono@math.uni-koeln.de}
\author{Larry Rolen}
\address{Department of Mathematics, 1420 Stevenson Center, Vanderbilt University, Nashville, TN 37240}
\email{larry.rolen@vanderbilt.edu}

\begin{document}


\begin{abstract}
Recently, Mertens, Ono, and the third author studied mock modular analogues of Eisenstein series. Their coefficients are given by small divisor functions, and have shadows given by classical Shimura theta functions. Here, we construct a class of small divisor functions $\sdf$ and prove that these generate the holomorphic part of polar harmonic Maa{\ss} forms of weight $\frac{3}{2}$. To this end, we essentially compute the holomorphic projection of mixed harmonic Maaß forms in terms of Jacobi polynomials, but without assuming the structure of such forms. Instead, we impose translation invariance and suitable growth conditions on the Fourier coefficients. Specializing to a certain choice of characters, we obtain an identity between $\sdft$ and Hurwitz class numbers, and ask for more such identities. Moreover, we prove $p$-adic congruences of our small divisor functions when $p$ is an odd prime. If $\chi$ is non-trivial we rewrite the generating function of $\sdf$ as a linear combination of Appell-Lerch sums and their first two normalized derivatives. Lastly, we offer a connection of our construction to meromorphic Jacobi forms of index $-1$ and false theta functions.
\end{abstract}

\subjclass[2020]{11F37, 11F30, 11F03}

\keywords{Appell-Lerch sums, Harmonic Maa{\ss} forms, Holomorphic projection, Hurwitz class numbers, Jacobi polynomials, Partial theta functions, Shimura theta function, Small divisor function}

\thanks{The research conducted by the first author for this paper is supported by the Pacific Institute for the Mathematical Sciences (PIMS). The research and findings may not reflect those of the Institute.}

\maketitle

\section{Introduction and statement of results}
A recent paper by Mertens, Ono, and the third author \cite{mertens2019mock} defined and investigated a new type of mock modular form. Their construction is motivated by work of Hecke \cite{hecke}, whose results imply that the functions
\begin{align*}
\frac{1}{2}L\left(1-k, \phi\right) + \sum_{n \geq 1} \left(\sum_{d \mid n} \phi(d)d^{k-1}\right)q^n, \qquad \sum_{n \geq 1} \left(\sum_{d \mid n} \phi\left(\frac{n}{d}\right)d^{k-1}\right)q^n
\end{align*}
are holomorphic weight $k$ modular forms on $\Gamma_0(N)$ with Nebentypus $\phi$ if $k > 2$, where $\phi$ is any primitive Dirichlet character of modulus $N$ satisfying $\phi(-1) = (-1)^k$. Here and throughout, we let $\tau = u+iv \in \H$ and $q \coloneqq \mathrm{e}^{2\pi i \tau}$. The notation $L(s, \phi)$ refers to the Dirichlet $L$-function of $\phi$. Mertens, Ono, and the third author focussed on the case $k=2$, and to this end defined a different class of twisted and restricted versions of classical divisor sums 
\begin{align*}
\sigma_{k-1}(n) \coloneqq \sum_{d \mid n} d^{k-1}.
\end{align*}
Since $\sigma_{k-1}(n)$ is a Fourier coefficient of the classical holomorphic Eisenstein series $E_k$ for even $k \geq 2$, they called these ``mock modular Eisenstein series''. Following the setting of \cite{mertens2019mock}, let $\psi$ be any non-trivial Dirichlet character of modulus $\Mp > 1$, define the set of admissible ``small'' divisors
\begin{align*}
D_n \coloneqq \left\{ d \mid n \ \colon 1 \leq d \leq \frac{n}{d} \text{ and } d \equiv \frac{n}{d} \pmod*{2}  \right\},
\end{align*}
and the \textit{small divisor function}
\begin{align*}
\sol (n) \coloneqq \sum_{d \in D_n} \psi\left( \frac{\left(\frac{n}{d}\right)^2 -d^2 }{4}\right) d.
\end{align*}
In addition, let
\begin{align*}
\lambda_{\psi} \coloneqq \frac{1-\psi(-1)}{2} \in \{0,1\}
\end{align*}
depending on the parity of $\psi$, let $\chi_{-4}$ be the unique odd character of modulus $4$, and we recall Shimura's theta function
\begin{align*}
\tp(\tau) \coloneqq \frac{1}{2}\sum_{n \in \Z} \psi(n)n^{\lambda_{\psi}}q^{n^2}.
\end{align*}
Then the main result of \cite{mertens2019mock} states that the function
\begin{align*}
\Ec^+(\tau) \coloneqq \frac{1}{\tp(\tau)}\left(\alpha_{\psi}E_2(\tau) + \sum_{n \geq 1} \sol(n) q^n \right)
\end{align*}
can be completed to a \textit{polar} harmonic Maa{\ss} form of weight $\frac{3}{2}-\lambda_{\psi}$ on $\Gamma_0(4\Mp^2)$ with Nebentypus $\overline{\psi} \cdot \chi_{-4}^{\lambda_{\psi}}$, where $\alpha_{\psi}$ is an implicit constant to ensure a certain growth condition. That is, it has the transformation and analytic properties of a harmonic Maa{\ss} form, but it may have poles on the upper half plane arising from the $\tp(\tau)$ in the denominator. Similar ideas have been utilized for specific examples before by Andrews, Rhoades, Zwegers \cite{anrhzwe}, and by Bringmann, Kane, Zwegers \cite{brikazwe} for instance. Additionally, the authors of \cite{mertens2019mock} presented some special choices of $\psi$ where their polar harmonic Maa{\ss} forms happen to have no poles on $\H$, and offer a $p$-adic property of $\Ec^+$ for primes $p > 3$. 

A natural question is whether there are more classes of small divisor functions for which a similar phenomenon to the setting of \cite{mertens2019mock} appears. Our two main results give such generalizations. We let $\chi$ be a second Dirichlet character of modulus $\Mc$ and define our small divisor function by
\begin{align*}
\sdf (n) &\coloneqq \sum_{d \in D_n} \chi\left(\frac{\frac{n}{d}-d}{2}\right)\psi\left(\frac{\frac{n}{d}+d}{2}\right)d^2.
\end{align*}
We stipulate that $\psi$ is odd and fixed throughout, and thus omit the dependency of $\sdf$ on $\psi$. We moreover define $\Fc(\tau) \coloneqq \Fcp(\tau) + \Fcm(\tau)$, where
\begin{align*}
\Fcp(\tau) \coloneqq \frac{1}{\tp(\tau)}\sum_{n \geq 1}\sdf(n)q^n, \qquad
\Fcm(\tau) \coloneqq \frac{i}{\pi \sqrt{2}} \int_{-\overline{\tau}}^{i\infty}\frac{\tc\left(w\right)}{\left(-i\left(w+\tau\right)\right)^{\frac{3}{2}}} dw.
\end{align*}
We obtain the following result.
\begin{thm}\label{thm:main}
If $\chi$ is even and non-trivial then the function $\Fc$ is a polar harmonic Maa{\ss} form of weight $\frac{3}{2}$ on $\Gamma_0\left(\lcm\left(4\Mc^2,4\Mp^2\right)\right)$ with Nebentypus $\overline{\chi} \cdot \left(\psi \cdot \chi_{-4}\right)^{-1}$. Its shadow is given by $\frac{1}{2\pi} \theta_{\overline{\chi}}$.
\end{thm}

Including the possibility of $\chi = \mathbbm{1}$, where a constant term arises, requires adjustments either in the holomorphic part $\Fcp$ or in the non-holomorphic part $\Fcm$. On one hand, we may subtract the arising constant term from $\Fcm$ again. The shadow of $\Fc$ would be given by the partial theta function
\begin{align*}
\frac{1}{2\pi}\sum_{n \geq 1} q^{n^2} = \frac{1}{2\pi}\ti(\tau) - \frac{1}{4\pi},
\end{align*}
and hence would not be modular anymore. In \cite{briro16} it is proved that all one-dimensional partial theta functions are (strong) quantum modular forms, which were first introduced by Zagier \cite{zagierquant}. Furthermore, such partial theta functions are related to Appell-Lerch sums of level $2$, to meromorphic Jacobi forms, and closely to false theta functions as well. An exposition on the former two connections was given by Bringmann, Zwegers, and the third author in \cite{brirozwe16}. To state their results, let $\zeta \coloneqq \mathrm{e}^{2 \pi i z}$, and
\begin{align*}
\vartheta(z;\tau) &\coloneqq \sum_{n \in \Z} q^{\frac{n^2}{2}}\zeta^n = -i\zeta^{-\frac{1}{2}} q^{\frac{1}{8}}\prod_{j=0}^{\infty}\left(1-q^{j+1}\right)\left(1-\zeta q^j\right)\left(1-\zeta^{-1}q^{j+1}\right)
\end{align*}
be the standard Jacobi theta function of index and weight $\frac{1}{2}$. The second equality is the Jacobi triple product identity, from which we deduce that $\vartheta(z;\tau)$ has zeros precisely in $\Z\tau+\Z$ as a function of $z$ and all zeros are simple. Moreover, letting $\ell \in \N$, the Appell-Lerch sum of level $\ell$ is defined by
\begin{align*}
A_{\ell}(w,z;\tau) \coloneqq \mathrm{e}^{\pi i \ell w}\sum_{n \in \Z} \frac{(-1)^{\ell n}q^{\ell\frac{n(n+1)}{2}}\mathrm{e}^{2\pi i n z}}{1-\mathrm{e}^{2\pi i w}q^n}.
\end{align*}
The results of \cite{brirozwe16} can be specialized to our setting and read as follows.
\begin{prop}[\protect{\cite[Corollaries 1.2 and 1.5]{brirozwe16}}] \label{Prop:parthetalerch}
\begin{enumerate} [label=(\roman*),  wide, labelwidth=!, labelindent=0pt]
\item We have
\begin{align*}
\sum_{n \geq 1} q^{n^2} = -\int_{-\frac{1}{2}}^{\frac{1}{2}} \mathrm{e}^{2\pi i (2t)} A_2\left(t-\frac{\tau}{2},0;\tau\right) dt.
\end{align*}
\item Let $f_n(\tau)$ be defined by the expansion \cite[eq.\ (2.5)]{brirozwe16}
\begin{align*}
\frac{1}{\vartheta(z;\tau)^2} \eqqcolon \sum_{n \geq -2} f_n(\tau) (2\pi i z)^n.
\end{align*}
Then it holds that
\begin{align*}
f_{-1}(\tau)\sum_{n \geq 1} q^{n^2} + 2f_{-2}(\tau)\sum_{n \geq 1} nq^{n^2} = \int_{-\frac{1}{2}}^{\frac{1}{2}} \frac{\mathrm{e}^{2\pi i (2t)}}{\vartheta\left(t-\frac{\tau}{2};\tau\right)^2}dt,
\end{align*}
and the functions $f_{-1}(\tau)$, $f_{-2}(\tau)$ both are known to be quasimodular forms (compare \cite[Theorem 3.2]{eiza}, \cite[Corollary 2.36]{thebook}, \cite[p.\ 8]{brirozwe16}).
\end{enumerate}
\end{prop}
Zwegers provided the non-holomorphic completion of $A_{\ell}$ to a two-variable Jacobi form of weight $1$ and matrix index $\left(\begin{smallmatrix} -\ell & 1 \\ 1 & 0 \end{smallmatrix}\right) $, see \cite[Theorem 4]{zwegers19}. Moreover, a third perspective arises from the close relation between partial and false theta functions. Bringmann and Nazaroglu described the completion of false theta functions to functions with certain Jacobi transformation properties in a recent paper \cite{brinaz}. In particular, their result includes the completion for
\begin{align*}
\sum_{n \in \Z} \textnormal{sgn}(n)q^{\frac{n^2}{2}}\mathrm{e}^{2 \pi i n z},
\end{align*}
where $\textnormal{sgn}(0) \coloneqq 0$.

On the other hand, we may compensate for the additional constant term by adjusting the holomorphic part $\Fcp$. Being more precise, we define $\Gc(\tau) \coloneqq \Gcp(\tau) + \Gcm(\tau)$, where
\begin{align*}
\Gcp(\tau) &\coloneqq \frac{1}{\tp(\tau)}\left(\frac{1}{2}\sum_{n \geq 1} \psi(n)n^2q^{n^2} + \sum_{n \geq 1}\sdft(n)q^n\right), \\
\Gcm(\tau) &\coloneqq \frac{i}{\pi \sqrt{2}} \int_{-\overline{\tau}}^{i\infty}\frac{\ti\left(w\right)}{\left(-i\left(w+\tau\right)\right)^{\frac{3}{2}}} dw.
\end{align*}
It turns out that the strategy of the proof of Theorem \ref{thm:main} still applies, enabling us to complete the picture (regarding even $\chi$) by the following result.
\begin{thm} \label{thm:varmain}
The function $\Gc$ is a polar harmonic Maa{\ss} form of weight $\frac{3}{2}$ on $\Gamma_0\left(4\Mp^2\right)$ with Nebentypus $\left(\psi \cdot \chi_{-4}\right)^{-1}$. Its shadow is given by $\frac{1}{2\pi} \ti$.
\end{thm}

Rouse and Webb showed in \cite{rouwe} that the only modular forms on $\Gamma_0(N)$ with integer Fourier coefficients and no zeros on $\H$ are eta quotients. In addition, Mersmann and Lemke-Oliver completed the classification of theta functions which may be written as such an eta quotient. We cite their results in the formulation of \cite[Theorem 1.2]{mertens2019mock}.
\begin{thm}[\protect{\cite{etatheta}, \cite{mersthesis}}]\label{thm:MLO}
The only nontrivial primitive characters $\psi$ for which $\tp$ is an eta quotient are contained in the set of Kronecker characters
\begin{align*}
\Psi \coloneqq \left\{ \left(\frac{D}{\cdot}\right) \colon D \in \{-8,-4,-3,2,12,24 \} \right\}.
\end{align*}
\end{thm}
Combining Theorems \ref{thm:main}, \ref{thm:varmain}, and \ref{thm:MLO} immediately yields the following corollary. (Recall that there is no odd character of modulus $2$.)
\begin{cor} \label{cor:mocktheta}
If $\psi \in \Psi \setminus \left\{\left(\frac{2}{\cdot}\right)\right\}$ is odd then $\Fcp$ and $\Gcp$ are mock theta functions.
\end{cor}

The result of Theorem \ref{thm:MLO} motivates us to investigate the odd choices $\psi \in \Psi \setminus \left\{\left(\frac{2}{\cdot}\right)\right\}$ in greater detail. To this end, let $H(n)$ be the Hurwitz class number, counting the weighted number of classes of positive definite binary quadratic forms of discriminant $-n$. Phrased in todays terminology, Zagier discovered in \cite{zagiereis} that 
\begin{align*}
\Hcp(\tau) \coloneqq -\frac{1}{12} + \sum_{n \geq 1} H(n)q^n.
\end{align*}
can be completed to a harmonic Maa{\ss} form $\mathcal{H}(\tau) \coloneqq \Hcp(\tau) -\frac{1}{4}\Gcm(\tau)$ of weight $\frac{3}{2}$ on $\Gamma_0(4)$. To see this, one may rewrite $\Gcm$ as in Lemma \ref{lem:fmrew} below, and compare \cite[Theorem 6.3]{thebook} for instance. The function $\mathcal{H}$ is then often called \textit{Zagier's weight $\frac{3}{2}$ non-holomorphic Eisenstein series}. A standard computation using the Sturm bound yields the following.

\begin{cor} \label{cor:hurwitz}
Let $\psi = \chi_{-4}$. Then we have
\begin{align*}
\sum_{n \geq 1}\sdft(8n)q^{8n} = -4\tp(\tau)\sum_{n \geq 1} H(8n-1)q^{8n-1}.
\end{align*}
Or in other words, by definition of $\psi=\chi_{-4}$ and $\tp$,
\begin{align*}
\sdft(8n) = 4 \sum_{\substack{j \geq 1 \\ (2j-1)^2 < 8n}} (-1)^{j}(2j-1)H\left(8n-(2j-1)^2\right).
\end{align*}
\end{cor}

In addition, we note that both the coefficients of $\tp\Hcp$ and the values of $\sdft$ grow at most polynomially. Based on the observations from the preceeding discussion we inquire the following.
\begin{que} 
For every $\psi \in \Psi \setminus \left\{\left(\frac{2}{\cdot}\right)\right\}$, do there exist numbers $0 \neq C,t \in \Q$, such that
\begin{align*}
C \cdot \Gcp(t\tau)
\end{align*}
generates a linear combination of Hurwitz class numbers on some arithmetic progression?
\end{que}

In the course of proving Theorem \ref{thm:main} and Theorem \ref{thm:varmain} we compute the holomorphic projection of a product similar to the structure of mixed harmonic Maa{\ss} forms. However, we just rely on translation invariance and impose suitable growth conditions on the coefficients to ensure convergence. The resulting expression can be written in terms of a particular class of Jacobi polynomials $\Pc_{r}^{(a,b)}$ (sometimes also referred to as \textit{hypergeometric polynomials}.). For $r \in \N_0$, these polynomials are defined by
\begin{align*}
\Pc_{r}^{(a,b)}(z) &\coloneqq \frac{\Gamma(a+r+1)}{r! \ \Gamma(a+b+r+1)} \sum_{j=0}^{r} {r \choose j} \frac{\Gamma(a+b+r+j+1)}{\Gamma(a+j+1)}\left(\frac{z-1}{2}\right)^{j} \\
&= \frac{\Gamma(a+r+1)}{r! \ \Gamma(a+1)} {}_2F_1\left(-r,a+b+r+1,a+1,\frac{1-z}{2}\right),
\end{align*}
which can be found in \cite[item 8.962]{table} for example. Here, ${}_2F_1$ denotes the usual Gau{\ss} hypergeometric function. Then we have the following result.
\begin{prop}\label{Prop:pimixed}
Let $k_{f} \in \R \setminus \N$, $k_g \in \R \setminus \left(-\N\right)$, such that $\kappa \coloneqq k_{f}+k_g \in \Z_{\geq 2}$. Let $\alpha(m)$, $\beta(n)$ be two complex sequences, and define\footnote{The function $\Gamma(s,z)$ denotes the incomplete Gamma function, which will be introduced in Section \ref{sec:prel}.}
\begin{align*}
f(\tau) \coloneqq \sum_{m \geq 1} \alpha(m)m^{k_{f}-1}\Gamma(1-k_{f},4\pi mv)q^{-m}, \qquad g(\tau) \coloneqq \sum_{n \geq 1} \beta(n)q^{n}.
\end{align*}
Suppose that 
\begin{enumerate}
\item the function $(fg)(r+iv)$ grows at most polynomially as $v \searrow 0$, where $r \in \Q$, and that
\item the function $(fg)(iv)$ grows at most polynomially as $v \nearrow \infty$.
\end{enumerate}
Then the weight $\kappa$ holomorphic projection of $fg$ is given by
\begin{align*}
\pi_{\kappa}\left({f}g\right)(\tau) = -\Gamma(1-k_{f}) \sum_{m \geq 1} \sum_{n-m \geq 1} \alpha(m) \beta(n)\left(n^{k_{f}-1}\Pc_{\kappa-2}^{(1-k_{f},1-\kappa)}\left(1-2\frac{m}{n}\right)-m^{k_{f}-1}\right)q^{n-m}.
\end{align*}
\end{prop}

We provide two proofs of this result, which correspond to either definition of the Jacobi polynomials. The first one relies on identities of the Gau{\ss} hypergeometric function, while the second one relies on two Lemmas from \cite{mertens2016}.

\begin{rmks}
\begin{enumerate}
\item Assuming the framework of mixed harmonic Maa{\ss} forms, some variants of this result appear in \cite[Theorem 3.5]{holopro}, and \cite[Theorem 4.6]{mertens2016}. Moreover, if $k_{g} = 2-k_{{f}}$, then $\Pc_0^{a,b}(z) = 1$. Therefore, working with regularized holomorphic projection, our result includes \cite[Proposition 2.1]{mertens2019mock}. 
\item Note that the summation conditions imply $-1 < 1-2\frac{m}{n} < 1$. The asymptotic behavior of the Jacobi polynomials inside $(-1,1)$ is well known and can be found in \cite[item 8.965]{table} for instance.
\item One may choose various other special values of half integral $k_f$, $k_g$, which simplify the Jacobi polynomial and then the whole factor 
\begin{align*}
n^{k_{f}-1}\Pc_{\kappa-2}^{(1-k_{f},1-\kappa)}\left(1-2\frac{m}{n}\right)-m^{k_{f}-1}.
\end{align*}
This idea leads to other choices of polynomials $P\left(\frac{n}{d},d\right) \in \Q[X,Y]$ than $d^2$ in the definition of $\sdf$ such that 
\begin{align*}
&\pi_{\kappa}\left(\sum_{n \geq 1} \sum_{d \in D_n}\chi\left( \frac{\frac{n}{d}-d}{2}\right)\psi\left(\frac{\frac{n}{d}+d}{2}\right)P\left(\frac{n}{d},d\right)  q^n \right. \\
&\hspace{3em} \left. + \sum_{n \geq 1} \sum_{m \geq 1} \alpha\left(m^2\right)\beta\left(n^2\right)m^{2(k_{f}-1)}\Gamma(1-k_{f},4\pi m^2v)q^{n^2-m^2}\right) = 0.
\end{align*}
We demonstrate during the proofs of Theorem \ref{thm:main} and \ref{thm:varmain} how to rewrite the corresponding generating function
\begin{align*}
\sum_{n \geq 1} \sum_{d \in D_n}\chi\left( \frac{\frac{n}{d}-d}{2}\right)\psi\left(\frac{\frac{n}{d}+d}{2}\right)P\left(\frac{n}{d},d\right)
\end{align*}
to obtain a choice of $P\left(\frac{n}{d},d\right) \in \Q[X,Y]$, which matches the factor involving the Jacobi polynomial.
\end{enumerate}
\end{rmks}

A second application of Proposition \ref{Prop:pimixed} arises from $p$-adic properties of $\Fcp$ and $\Gcp$. Mertens, Ono and the third author proved such a property for their mock modular Eisenstein series $\Ec^+$, c.\ f.\ \cite[Theorem 1.4]{mertens2019mock}. More precisely, the idea is to inspect the iterated action of the $U$-operator
\begin{align*}
\left(\sum_{n \gg -\infty}\alpha(n)q^n\right) \Big\vert \ U(p) \coloneqq \sum_{n \gg -\infty}\alpha(pn)q^n
\end{align*}
on $\tp\left(p^{2a}\tau\right)\Ec^+(\tau)$ for every $a \in \N$ and $p > 3$ prime. (The notation $\sum_{n \gg -\infty}$ is explained in Lemma \ref{lem:mafu}.) Then they show that this is congruent to some meromorphic modular form of weight $2$. In our case Theorems \ref{thm:main} and \ref{thm:varmain} imply that the products $\tp(\tau)\Fc(\tau)$ and $\tp(\tau)\Gc(\tau)$ are modular of weight $3$ with Nebentypus $\overline{\chi}$ or trivial Nebentypus respectively. Therefore, we find a different result.
\begin{thm} \label{thm:padicthm}
Let $a$, $b$, $p \in \N$ and suppose that $p$ is an odd prime. Then we have
\begin{align*}
\left(\tp\left(p^{2a}\tau\right)\Fcp(\tau)\right) \Big\vert \ U\left(p^b\right) \equiv 0 \quad \pmod*{p^{\textnormal{min}(a,b)}},
\end{align*}
and
\begin{align*}
\left(\tp\left(p^{2a}\tau\right)\Gcp(\tau)\right) \Big\vert \ U\left(p^b\right) \equiv 0 \quad \pmod*{p^{\textnormal{min}(a,b)}}.
\end{align*}
\end{thm}

The third remark on page 3 in \cite{mertens2019mock} states that ``the generating function of $\sol$ can be given in terms of Appell-Lerch sums as studied by Zwegers'' in \cite{zwegersthesis}, \cite{zwegers19} (see also \cite[Lemma 2]{mertens2014}). This remark applies verbatim to the generating function of $\sdf$ as well, and we present a strategy which applies to both generating functions. Let $D_z \coloneqq \frac{1}{2\pi i} \frac{d}{dz}$, giving
\begin{align*}
(D_z^jA_{\ell})(w,z,\tau) = \mathrm{e}^{\pi i \ell w}\sum_{n \in \Z} n^j\frac{(-1)^{\ell n}q^{\ell\frac{n(n+1)}{2}}\mathrm{e}^{2\pi i n z}}{1-\mathrm{e}^{2\pi i w}q^n}.
\end{align*}
for every integer $j \geq 0$. Then we have the following result.
\begin{prop} \label{Prop:alprop}
Suppose that $\chi$ is non-trivial and even. Additionally assume $\Mp \mid \Mc$. Then we have that
\begin{align*}
& \sum_{n \geq 1} \sdf(n) q^n \\
&= \frac{1}{2} \sum_{b=1}^{\Mc-1} \chi(b)\sum_{c=0}^{\Mc-1}\psi(b+c) q^{c(c+2b-\Mc)} \left(\Mc D_z + c\right)^2A_1\left(2\Mc c\tau,z,2\Mc^2\tau\right)\Big\vert_{z = \left(2(b+c)-\Mc\right)\Mc\tau +\frac{1}{2}}.
\end{align*}
\end{prop}

Lastly, it is also likely that the function $\sdf$ can be viewed as a Siegel theta lift in the following way. A recent paper of Bruinier and Schwagenscheidt \cite{bruinier2020theta} investigated the Siegel theta lift on Lorentzian lattices, and its connection to coefficients of mock theta functions. In an isotropic lattice of signature $(1,1)$, for example, they obtained a formula for the Siegel theta lift of a particular weakly holomorphic modular form evaluated at a certain point in the Grassmanian in terms of the sum
\begin{align*}
\sum_{\substack{r \in \Z \\ r \equiv 1 \pmod*{2}}} H(4m-r^2).
\end{align*}
Furthermore, in \cite{alfesneumann2020cycle}, Alfes-Neumann, Bringmann, Schwagenscheidt, and the first author considered the same lift on a lattice of signature $(1,2)$, but with the inclusion of an iterated Maass raising operator acting on the weakly holomorphic modular form. There, we obtained an expression of the form (see \cite[Example 1.2]{alfesneumann2020cycle})
\begin{align*}
\sum_{\substack{n,m \in \Z \\ n \equiv D \pmod*{2}}} \left(4D-10n^2-10m^2\right) H\left(D-n^{2}-m^{2}\right).
\end{align*}
The quadratic form in the Hurwitz class number is explained by the signature being $(1,2)$ in this case, and the addition of the polynomial is a consequence of the iterated Maass raising operators.

In view of Corollary \ref{cor:hurwitz}, our smallest divisor function seems to lie at the interface of these two situations, i.e. it is natural to expect that it can be realized as a theta lift in signature $(1,1)$ where the lift includes iterated Maass raising operators. Perhaps this phenomenon can also be extended to further classes of small divisor functions -  the theta lift construction allows one to choose many examples of mock theta functions of appropriate weight, not just the generating function for Hurwitz class numbers.

\bigskip

The paper is organized as follows. The upcoming section briefly recaps the overall framework and collects some basic properties, which we require later. The aim of Section \ref{sec:pimixedproof} is to present two proofs of Proposition \ref{Prop:pimixed}. Combining the results of Section \ref{sec:prel} and \ref{sec:pimixedproof} enables us to prove our two main Theorems in Section \ref{sec:mainproofs}. Some modifications of the computations during Section \ref{sec:mainproofs} prove Theorem \ref{thm:padicthm}, and we devote Section \ref{sec:remthms} to them. Finally, the purpose of Section \ref{sec:alparproof} is to prove the remaining Propositions \ref{Prop:parthetalerch} and \ref{Prop:alprop}.

\section*{Acknowledgements} The authors would like to thank Michael Mertens particularly, whose comments significantly improved this paper. Furthermore, we would like to thank Kathrin Bringmann as well as the anonymous referee for many valuable comments on an earlier version of this paper. The research conducted by the first author for this paper is supported by the Pacific Institute for the Mathematical Sciences (PIMS). The research and findings may not reflect those of the Institute.
The third author is grateful for support from a grant from the Simons Foundation (853830, LR), support from a Dean’s Faculty Fellowship from Vanderbilt University, and to the Max Planck Institute for Mathematics in Bonn for its hospitality and financial support.

\section{Preliminaries} \label{sec:prel}

\subsection{Growth conditions and modular forms}
To describe our various modular objects, we first require some terminology on growth conditions.
\begin{defn}
Let
\begin{align*}
f(\tau) \coloneqq \sum_{n \in \Z} c_f(n) q^n
\end{align*}
with some complex coefficients $c_f(n)$. Then we say that
\begin{enumerate}[label=(\roman*), wide, labelwidth=!, labelindent=0pt]
\item the function $f$ is {\it holomorphic} at $i\infty$ if $c_f(n) = 0$ for every $n < 0$.
\item the function $f$ is of {\it moderate growth} at $i\infty$ if $f \in O\left(v^m\right)$ as $v \to \infty$ for some $m \in \N$. In other words $f$ grows at most polynomially.
\item the function $f$ is of \textit{governable growth} at $i\infty$ if there exists $ P_f \in \C\left[q^{-1}\right]$ such that for some $\delta>0$ we have
\begin{align*}
 f(\tau) - P_f(\tau) \in O\left(\mathrm{e}^{-\delta v}\right), \quad v \rightarrow \infty.
\end{align*}
The polynomial $P_f$ is called the \textit{principal part} of $f$. Equivalently, $f$ is permitted to have a pole at $i\infty$.
\item $f$ is of {\it linear exponential growth} if $f(\tau) \in O\left(\mathrm{e}^{\delta v}\right)$ as $v \to \infty$ for some $\delta > 0$.
\end{enumerate}
\end{defn}

These conditions can be phrased at all other cusps via suitable scaling matrices. To define the slash-operator, we take the principal branch of the holomorphic square root throughout.
\begin{defn}
Let $k \in \frac{1}{2}\Z$, $\phi$ be a Dirichlet character, and $\gamma =  \left(\begin{smallmatrix} a & b \\ c & d\end{smallmatrix}\right) \in \slz$. Then the \textit{slash operator} is defined as
\begin{align*}
\left(f\vert_k\gamma\right)(\tau) \coloneqq \begin{cases}
\phi(d)^{-1}(c\tau+d)^{-k} f(\gamma\tau) & \text{if} \ k \in \Z, \\
\phi(d)^{-1}\left(\frac{c}{d}\right)\varepsilon_d^{2k}(c\tau+d)^{-k} f(\gamma\tau) & \text{if} \ k \in \frac{1}{2}+\Z,
\end{cases}
\end{align*}
where $\left(\frac{c}{d}\right)$ denotes the extended Legendre symbol, and
\begin{align*}
\varepsilon_d \coloneqq \begin{cases}
1 & \text{if} \ d \equiv 1 \pmod*{4}, \\
i & \text{if} \ d \equiv 3 \pmod*{4}. \\
\end{cases}
\end{align*}
for odd integers $d$.
\end{defn}

For the sake of completeness, we define the variants of modular forms appearing in this paper.
\begin{defn}
Let $f \colon \H \to \C$ be a function, $\Gamma \leq \slz$ be a subgroup, $\phi$ be a Dirichlet character, and $k \in \frac{1}{2}\Z$. Then we say that
\begin{enumerate}[label=(\roman*), wide, labelwidth=!, labelindent=0pt]
\item The function $f$ is a {\it modular form} of weight $k$ on $\Gamma$ with Nebentypus $\phi$ if
\begin{enumerate}[label=(\alph*), wide, labelwidth=!]
\item for every $\gamma \in \Gamma$ and every $\tau \in \H$ we have $\left(f\vert_k\gamma\right)(\tau) = f(\tau)$,
\item $f$ is holomorphic on $\H$,
\item $f$ is holomorphic at every cusp.
\end{enumerate}
We denote the vector space of functions satisfying these conditions by $M_k(\Gamma, \phi)$. 
\item If in addition $f$ vanishes at every cusp, then we call $f$ a \textit{cusp form}. The subspace of cusp forms is denoted by $S_k(\Gamma, \phi)$.
\item If $f$ satisfies the conditions (a) and (b) from (i) and is allowed to have a pole at one or more cusps, then we call $f$ a {\it weakly holomorphic modular form} of weight $k$ on $\Gamma$ with Nebentypus $\phi$. The vector space of such functions is denoted by $M_k^!(\Gamma, \phi)$.
\end{enumerate}
\end{defn}

Furthermore, we recall the following fact which we require throughout.
\begin{lemma}
The following are true
\begin{align*}
\tp \in \begin{cases}
M_\frac{1}{2}(\Gamma_0(4M_\psi^2), \psi) & \text{if} \ \lambda_{\psi} = 0, \\
S_\frac{3}{2}(\Gamma_0(4M_\psi^2), \psi\cdot\chi_{-4}) & \text{if} \ \lambda_{\psi} = 1.
\end{cases}
\end{align*}
\end{lemma}
A proof can be found in \cite[Theorem 10.10]{iwaniec} for instance.

Finally, Corollary \ref{cor:hurwitz} follows directly from the following result.
\begin{lemma}[Sturm bound]
Let $f, g \in M_k\left(\Gamma_0(N), \phi\right)$, $k > 1$, and
\begin{align*}
m \coloneqq \left[\slz\colon\Gamma_0(N)\right] = N\prod_{p\mid N}\left(1+\frac{1}{p}\right).
\end{align*}
Define
\begin{align*}
B \coloneqq \floor*{\frac{km}{12}},
\end{align*}
and denote by $c_f(n)$, $c_g(n)$ the coefficients of the $q$-expansions of $f$ and $g$ respectively. If $c_f(n) = c_g(n)$ for all $n \leq B$ then $f=g$.
\end{lemma}
A proof can be found in \cite[Corollary 9.20]{stein}.

\subsection{Harmonic Maa{\ss} forms and shadows}
We define our main objects of interest.
\begin{defn} \label{def:maassforms}
Let $k \in \frac{1}{2}\Z$, and choose $N \in \N$ such that $4 \mid N $ whenever $k \not\in \Z$. Let $\phi$ be a Dirichlet character of modulus $N$. 
\begin{enumerate}[label=(\roman*), wide, labelwidth=!, labelindent=0pt]
\item A weight $k$ {\it harmonic Maa{\ss} form} on a subgroup $\Gamma_0(N)$ with Nebentypus $\phi$ is any smooth function $f \colon \H \to \C$ satisfying the following three properties:
\begin{enumerate}[label=(\alph*), wide, labelwidth=!]
\item For all $\gamma \in \Gamma_0(N)$ and all $\tau \in \H$ we have $\left(f\vert_k\gamma\right)(\tau) = f(\tau)$.
\item The function $f$ is harmonic with respect to the weight $k$ hyperbolic Laplacian on $\H$, explicitly
\begin{align*}
0 = \Delta_k f \coloneqq \left(-v^2\left(\frac{\partial^2}{\partial u^2}+\frac{\partial^2}{\partial v^2}\right) + ikv\left(\frac{\partial}{\partial u} + i\frac{\partial}{\partial v}\right)\right) f.
\end{align*}
\item The function $f$ has at most linear exponential growth at all cusps. 
\end{enumerate}
We denote the vector space of such functions by $H_k^!(\Gamma_0(N), \phi)$. 
\item If we restrict the growth condition in (c) to governable growth then the vector space of such forms is denoted by $H_k(\Gamma_0(N), \phi)$.
\item A polar harmonic Maa{\ss} form is a harmonic Maa{\ss} form with isolated poles on the upper half plane.
\end{enumerate}
\end{defn}

\begin{rmk}
If we restrict the growth condition (c) to moderate growth at all cusps, and allow arbitrary eigenvalues in (b), then $f$ is a classical \textit{Maa{\ss} wave form}.
\end{rmk}

During the following summary, we may assume that $\phi$ is trivial for simplicity, since the generalization to a nontrivial Nebentypus is immediate. 

Bruinier and Funke observed in \cite{brufu} that the Fourier expansion of a harmonic Maa{\ss} form\footnote{Be aware that their terminology refers to our harmonic Maa{\ss} forms as ``harmonic \textit{weak} Maa{\ss} forms'' instead.} naturally splits into two parts. One of them involves the incomplete Gamma function
\begin{align*}
\Gamma(s,z) \coloneqq \int_z^{\infty} t^{s-1}\mathrm{e}^{-t} dt,
\end{align*}
defined for $\re(s) > 0$ and $z \in \C$. It can be analytically continued in $s$ via the functional equation
\begin{align*}
\Gamma(s+1,z) = s\Gamma(s,z) + z^{s}\mathrm{e}^{-z},
\end{align*}
and has the asymptotic behavior
\begin{align*}
\Gamma(s,v) \sim v^{s-1}\mathrm{e}^{-v}, \quad \vert v \vert \to \infty 
\end{align*}
for $v \in \R$. We state their result in the formulation of \cite[Lemma 4.3]{thebook}.
\begin{lemma} \label{lem:mafu}
Let $k \in \frac{1}{2}\Z \setminus \{1\}$ and $f \in H_k^{!}(\Gamma_0(N))$. Then $f$ has a Fourier expansion of the shape 
\begin{align*}
f(\tau) = \sum_{n \gg -\infty} c_f^+(n) q^n + c_f^-(0)v^{1-k} + \sum_{\substack{n \ll \infty \\ n \neq 0}} c_f^-(n)\Gamma(1-k,-4\pi nv)q^n.
\end{align*}
In particular, if $f \in H_k(\Gamma_0(N))$ then $f$ has a Fourier expansion of the shape 
\begin{align*}
f(\tau) = \sum_{n \gg -\infty} c_f^+(n) q^n + \sum_{n < 0} c_f^-(n)\Gamma(1-k,-4\pi nv)q^n.
\end{align*}
The notation $\sum_{n \gg -\infty}$ abbreviates $\sum_{n \geq m_f}$ for some $m_f \in \Z$. The notation $\sum_{n \ll \infty}$ is defined analogously, and similar expansions hold at the other cusps.
\end{lemma}

We follow the following terminology from \cite[Definition 4.4]{thebook}.
\begin{defn}
We refer to the functions 
\begin{align*}
f^+(\tau) \coloneqq \sum_{n \gg -\infty} c_f^+(n) q^n, \qquad f^-(\tau) \coloneqq c_f^-(0)v^{1-k} + \sum_{\substack{n \ll \infty \\ n \neq 0}} c_f^-(n)\Gamma(1-k,-4\pi nv)q^n
\end{align*}
as the \textit{holomorphic part} of $f$ and to $f^-$ as its \textit{non-holomorphic part}. 
\end{defn}

In the same paper, Bruinier and Funke introduced the operator
\begin{align*}
\xi_k \coloneqq 2iv^k\overline{\frac{\partial}{\partial\overline{\tau}}} = iv^{k}\overline{\left(\frac{\partial}{\partial u} + i\frac{\partial}{\partial v}\right)}.
\end{align*}
We summarize its relevant properties.
\begin{lemma} \label{lem:shadowprop}
Let $f$ be a smooth function on $\H$. Then the $\xi$-operator satisfies the following properties.
\begin{enumerate}[label=(\roman*),  wide, labelwidth=!, labelindent=0pt]
\item We have $\xi_k (f) = 0$ if and only if $f$ is holomorphic.
\item The slash operator intertwines with $\xi_k$, i.\ e.\ , we have
\begin{align*}
\xi_k\left(f\vert_k\gamma\right) = \left(\xi_k f\right)\vert_{2-k}\gamma
\end{align*}
for every $\gamma \in \slz$ if $k \in \Z$ or $\gamma \in \Gamma_0(4)$ if $k \in \frac{1}{2}\Z \setminus \Z$ respectively.
\item The kernel of $\xi_k$ restricted to $H_k(\Gamma_0(N))$ or $H_k^{!}(\Gamma_0(N))$ is precisely the space $M_k^!(\Gamma_0(N))$ in both cases.
\item Let $f \in H_k^{!}(\Gamma_0(N))$. Assuming the notation of Lemma \ref{lem:mafu} we have
\begin{align*}
\xi_k f(\tau) = \xi_k f^-(\tau) = (1-k)c_f^-(0) -(4\pi)^{1-k} \sum_{n \gg -\infty} \overline{c_f^-(-n)}n^{1-k}q^n \in M_{2-k}^{!}\left(\Gamma_0\left(N\right)\right),
\end{align*} 
and in particular if $f \in H_k(\Gamma_0(N))$ then
\begin{align*}
\xi_k f(\tau) = -(4\pi)^{1-k} \sum_{n \geq 1} \overline{c_f^-(-n)}n^{1-k}q^n \in S_{2-k}\left(\Gamma_0\left(N\right)\right).
\end{align*}
In addition, we have $\xi_k \colon H_k^{!}(\Gamma_0(N)) \twoheadrightarrow M_{2-k}^{!}\left(\Gamma_0\left(N\right)\right)$.
\end{enumerate}
\end{lemma}
The first item is simply a reformulation of the Cauchy-Riemann equations and the second one is induced by the corresponding well known property for the Maa{\ss} lowering operator
\begin{align*}
L_k \coloneqq v^2\frac{\partial}{\partial\overline{\tau}} = \frac{1}{2} v^2\left(\frac{\partial}{\partial u} + i\frac{\partial}{\partial v}\right).
\end{align*}

We fix some more terminology, following \cite[Definition 5.16]{thebook} and the second remark afterwards.
\begin{defn}
\begin{enumerate}[label=(\roman*), wide, labelwidth=!, labelindent=0pt]
\item A function $f$ is called a \textit{mock modular form} if $f$ is the holomorphic part of a harmonic Maa{\ss} form for which $f^-$ is nontrivial. 
\item If $f \in H_k^!(\Gamma_0(N))$ then we refer to the form $\left(\xi_k f\right)$ as the \textit{shadow} of $f^+$.
\item In particular, $f$ is called a \textit{mock theta function} if $f$ is a mock modular form of weight $\frac{1}{2}$ or $\frac{3}{2}$, whose shadow is a linear combination of unary theta functions.
\end{enumerate}
\end{defn}


Moreover, we study the following objects, which were introduced first in \cite[Section 7.3]{damuza}. However, we follow the definition given in \cite[Section 13.2]{thebook}.
\begin{defn}
\begin{enumerate}[label=(\roman*), wide, labelwidth=!, labelindent=0pt]
\item A \textit{mixed harmonic Maa{\ss} form} of weight $(k_1,k_2)$ is a function $h$ of the shape
\begin{align*}
h(\tau) = \sum_{j=1}^n f_j(\tau)g_j(\tau),
\end{align*}
where $f_j \in H_{k_1}$ and $g_j \in M_{k_2}^!$ for every $j$. 
\item Analogously, a \textit{mixed mock modular form} of weight $(k_1,k_2)$ is a function $h$ of the shape
\begin{align*}
h(\tau) = \sum_{j=1}^n f_j(\tau)g_j(\tau),
\end{align*}
where each $f_j$ is a mock modular form of weight $k_1$ and $g_j \in M_{k_2}^!$ for every $j$. 
\end{enumerate}
\end{defn}

We extend the last result of Lemma \ref{lem:shadowprop} to mixed harmonic Maa{\ss} forms. It suffices to consider products involving the non-holomorphic part of a mixed harmonic Maa{\ss} form.
\begin{lemma} \label{lem:shadowmixed}
Let $k_f$, $k_g \in \R$, $\kappa \coloneqq k_f+k_g$, and let $\alpha(m)$, $\beta(n)$ be two complex sequences. Define
\begin{align*}
f(\tau) \coloneqq \sum_{m \geq 1} \alpha(m)m^{k_{f}-1}\Gamma(1-k_{f},4\pi mv)q^{-m}, \qquad g(\tau) \coloneqq \sum_{n \geq 1} \beta(n)q^{n}.
\end{align*}
Then 
\begin{align*}
\xi_{\kappa} \left(fg\right)(\tau) = -(4\pi)^{1-k_f} v^{k_{g}} \sum_{m \geq 1} \overline{\alpha(m)} q^m \sum_{n \geq 1}\overline{\beta(n)q^n}.
\end{align*}
\end{lemma}

\begin{proof}
We have
\begin{align*}
(fg)(\tau)  = \sum_{m \geq 1} \sum_{n \geq 1} \alpha(m)m^{k_f-1}\Gamma(1-k_f,4\pi m v)\beta(n)q^{n-m}
\end{align*}
and
\begin{align*}
\frac{\partial}{\partial v} \Gamma(a,v) = -v^{a-1}\mathrm{e}^{-v}.
\end{align*}
We compute that $\xi_{\kappa}\left(\Gamma(1-k_f,4\pi m v) q^{n-m}\right)$ equals
\begin{align*}
& iv^{\kappa}\left\{\Gamma(1-k_f,4\pi mv)\overline{2\pi i (n-m)q^{n-m}} \right. \\	
& \left. -i \left[-(4\pi m v)^{-k_f}\mathrm{e}^{-4\pi m v}4\pi m \overline{q^{n-m}} + \Gamma(1-k_f,4\pi mv)(-2\pi(n-m))\overline{q^{n-m}} \right]\right\} \\
=& -v^{k_{g}}(4\pi m)^{1-k_f}\mathrm{e}^{-4\pi m v}\overline{q^{n-m}} = -v^{k_{g}}(4\pi m)^{1-k_f} q^m \mathrm{e}^{-2\pi in\overline{\tau}},
\end{align*}
and infer
\begin{align*}
\left(\xi_{\kappa} \left(fg\right)\right)(\tau) = -(4\pi)^{1-k_f}v^{k_{g}} \sum_{m \geq 1} \sum_{n \geq 1} \overline{\alpha(m)\beta(n)}q^m\overline{q^n},
\end{align*}
as claimed.
\end{proof}

\subsection{Holomorphic projection}
We introduce the holomorphic projection operator. Its origin lies in the search for an operator which preserves the (regularized) Petersson inner product. Although this can be derived implicitly from the Riesz representation theorem, an explicit description comes in handy quite often.

\begin{defn}
Let $f \colon \H \to \C$ be a translation invariant function and $k \in \Z_{\geq 2}$. 
If $f$ has at most moderate growth towards the cusps then the weight $k$ \textit{holomorphic projection} of $f$ is defined by
\begin{align*}
\left(\pi_k f\right)(\tau) \coloneqq \frac{(k-1) (2i)^{k}}{4\pi} \int_{\H} \frac{f\left(x+iy\right)y^{k}}{\left(\tau-x+iy\right)^{k}} \frac{dxdy}{y^2},
\end{align*}
whenever the integral converges absolutely.
\end{defn}

Furthermore, we require the following result.
\begin{lemma}[Lipschitz summation formula]
For any $r \in \N$ we have that
\begin{align*}
\sum_{j \in \Z} \frac{1}{(w+j)^{r}} = \frac{(-2 \pi i)^r}{(r-1)!}\sum_{j \geq 1} j^{r-1}\mathrm{e}^{2\pi i j w}.
\end{align*}
\end{lemma}
A short proof is due to Zagier and can be found in \cite[p.\ 16]{the123}.
We summarize two further properties of the holomorphic projection operator, both of which are proven in \cite[Section 3]{brikazwe} for instance.
\begin{lemma}
Let $f \colon \H \to \C$ be a translation invariant function of moderate growth such that the integral defining $\pi_k f$ converges absolutely, and $k \in \Z_{\geq 2}$. Then $\pi_k$ enjoys the following properties.
\begin{enumerate}[label=(\roman*), wide, labelwidth=!, labelindent=0pt]
\item If $f$ is holomorphic then $\pi_k f = f$.
\item If $f$ is modular with some Nebentypus $\phi$ (but not necessarily holomorphic), then
\begin{align*}
\pi_k f \in \begin{cases}
M_k\left(\Gamma_0\left(N\right), \phi\right) & \text{if} \ k \in \Z_{\geq 3}, \\
M_2\left(\Gamma_0\left(N\right), \phi\right) \oplus  M_0\left(\Gamma_0\left(N\right), \phi\right)\cdot E_2 & \text{if} \ k = 2.
\end{cases}
\end{align*}
Therefore, the slash operator and $\pi_k$ commute if $k \geq 3$.
\end{enumerate}
\end{lemma}

\section{Two proofs of Proposition \ref{Prop:pimixed}} \label{sec:pimixedproof}

During the first proof of Proposition \ref{Prop:pimixed}, we appeal to the following technical results.
\begin{lemma}\label{lem: technical details}
\begin{enumerate}[label=(\roman*), wide, labelwidth=!, labelindent=0pt]
\item If $\re{(b)}, \re{(a+b)} > 0$ and $\re{(c+s)} > 0$ then 
\begin{align*}
\int_0^{\infty} \Gamma(a,cz)z^{b-1}\mathrm{e}^{-sz}dz = \frac{c^{a}\Gamma(a+b)}{b(c+s)^{a+b}} \ {}_{2}F_{1}\left(1,a+b,b+1;\frac{s}{s+c}\right). 
\end{align*}
\item The hypergeometric function ${}_{2}F_{1}$ satisfies
\begin{align*}
{}_{2}F_{1}(a, b, c; z) = (1-z)^{c-a-b} \ {}_{2}F_{1}(c-a, c-b, c; z),
\end{align*}
and
\begin{align*}
{}_{2}F_{1}(a, b, c; z) &= \frac{\Gamma(c)\Gamma(c-a-b)}{\Gamma(c-a)\Gamma(c-b)}{}_{2}F_{1}(a,b,a+b-c+1;1-z) \\
& \hspace*{1em} + (1-z)^{c-a-b} \frac{\Gamma(c)\Gamma(a+b-c)}{\Gamma(a)\Gamma(b)}{}_{2}F_{1}(c-a,c-b,c-a-b+1;1-z).
\end{align*}
\end{enumerate}
\end{lemma}
The first identity is \cite[item 6.455]{table}. Both hypergeometric transformations can be found in \cite[page 1008]{table}.

\begin{proof}[First proof of Proposition \ref{Prop:pimixed}]
The $q$-expansion of $\left(fg\right)(\tau)$ is given by
\begin{align*}
(fg)(\tau) &= \sum_{m \geq 1} \sum_{n \geq 1} \alpha(m)\beta(n)m^{k_f-1}\Gamma(1-k_f,4\pi m v)q^{n-m}.
\end{align*}
We see that $fg$ is translation invariant, and recall that it has moderate growth towards all cusps by assumption, so $\pi_{\kappa}\left(fg\right)$ exists. 
Hence, we need to calculate
\begin{align*}
\pi_{\kappa}\left(fg\right)(\tau) = \frac{(\kappa-1) (2i)^{\kappa}}{4\pi} \int_{\H} \frac{\left(fg\right)\left(x+iy\right)y^{\kappa}}{\left(\tau-x+iy\right)^{\kappa}} \frac{dxdy}{y^2}.
\end{align*}
The integral converges since $\kappa-2 \geq 0$, and converges absolutely if $\kappa > 2$. Using the translation invariance of $fg$, we rewrite the integral over $\H$ as
\begin{align*}
\int_{\H} \frac{\left(fg\right)\left(x+iy\right)y^{\kappa}}{\left(\tau-x+iy\right)^{\kappa}} \frac{dxdy}{y^2} = \int_{0}^{\infty} \int_{0}^{1} \left(fg\right)(x+iy)y^{\kappa-2} \sum_{j \in \Z} \frac{1}{(\tau-x+iy+j)^{\kappa}} dxdy,
\end{align*}
and consequently
\begin{multline*}
\pi_{\kappa}\left(fg\right)(\tau) = \frac{(\kappa-1) (2i)^{\kappa}}{4\pi} \sum_{m \geq 1} \sum_{n \geq 1} \alpha(m)m^{k_f-1} \beta(n) \\
 \times \int_{0}^{\infty} \int_{0}^{1} \Gamma(1-k_f,4\pi m y)y^{\kappa-2} \sum_{j \in \Z} \frac{1}{(\tau-x+iy+j)^{\kappa}} \mathrm{e}^{2\pi i (n-m) (x+iy)} dxdy.
\end{multline*}
By the Lipschitz summation formula and then Lemma \ref{lem: technical details} (i), we infer that $\pi_{\kappa}\left(fg\right)(\tau)$ equals
\begin{align*}
& \frac{(\kappa-1) (2i)^{\kappa}}{4\pi}\frac{(-2 \pi i)^{\kappa}}{(\kappa-1)!} \sum_{m \geq 1} \sum_{n \geq 1} \alpha(m)m^{k_f-1} \beta(n) \\
& \hspace{6em} \times \int_{0}^{\infty} \Gamma(1-k_f,4\pi m y)y^{\kappa-2} \int_{0}^{1} \sum_{j \geq 1} j^{\kappa-1}\mathrm{e}^{2\pi i j (\tau-x+iy)} \mathrm{e}^{2\pi i (n-m) (x+iy)} dxdy \\
= & \frac{(\kappa-1) (2i)^{\kappa}}{4\pi}\frac{(-2 \pi i)^{\kappa}}{(\kappa-1)!} \sum_{m \geq 1} \sum_{n-m \geq 1} \alpha(m)m^{k_f-1} \beta(n) (n-m)^{\kappa-1} \\
& \hspace{6em} \times \int_{0}^{\infty} \Gamma(1-k_f,4\pi m y)y^{\kappa-2} \mathrm{e}^{-4\pi (n-m) y} dy \ q^{n-m} \\
= & \frac{\Gamma(k_g)}{(\kappa-1)!} \sum_{m \geq 1} \sum_{n \geq m+1} \alpha(m)\beta(n) \frac{(n-m)^{\kappa-1}}{n^{k_g}} \ {}_{2}F_{1}\left(1,k_g,\kappa;1-\frac{m}{n}\right) q^{n-m}.
\end{align*}

Finally, we apply the hypergeometric transformations from Lemma \ref{lem: technical details} (ii). Explicitly,
\begin{align*}
& {}_{2}F_{1}\left(1,k_g,\kappa;1-\frac{m}{n}\right) = \frac{(\kappa-1)}{k_f-1} \ {}_{2}F_{1}\left(1,k_g,2-k_f;\frac{m}{n}\right) + \frac{\Gamma(\kappa)\Gamma(1-k_f)}{\Gamma(k_g)}\left(1-\frac{m}{n}\right)^{1-\kappa}\left(\frac{m}{n}\right)^{k_f-1} \\
&= \frac{(\kappa-1)}{k_f-1} \left(1-\frac{m}{n}\right)^{1-\kappa} \ {}_{2}F_{1}\left(1-k_f,2-\kappa,2-k_f;\frac{m}{n}\right) + \frac{\Gamma(\kappa)\Gamma(1-k_f)}{\Gamma(k_g)}\left(1-\frac{m}{n}\right)^{1-\kappa}\left(\frac{m}{n}\right)^{k_f-1}.
\end{align*}
Thus, we arrive at
\begin{align*}
& \frac{\Gamma(k_g)}{(\kappa-1)!}\frac{(n-m)^{\kappa-1}}{n^{k_g}} \ {}_{2}F_{1}\left(1,k_g,\kappa;1-\frac{m}{n}\right) \\
&= \frac{\Gamma(k_g)}{(\kappa-1)!}\left(n^{k_f-1}\frac{(\kappa-1)}{k_f-1} \ {}_{2}F_{1}\left(1-k_f,2-\kappa,2-k_f;\frac{m}{n}\right) + \frac{\Gamma(\kappa)\Gamma(1-k_f)}{\Gamma(k_g)} m^{k_f-1}\right) \\
&= -\Gamma(1-k_f)\left(n^{k_f-1}\frac{\Gamma(k_g)}{(\kappa-2)! \ \Gamma(2-k_f)} \ {}_2F_1\left(2-\kappa,1-k_f,2-k_f, \frac{m}{n}\right) - m^{k_f-1}\right) \\
&= -\Gamma(1-k_f)\left(n^{k_f-1}\Pc_{\kappa-2}^{(1-k_f,1-\kappa)}\left(1-2\frac{m}{n}\right) - m^{k_f-1}\right),
\end{align*}
and ultimately obtain
\begin{align*}
\pi_{\kappa}\left({f}g\right)(\tau) = -\Gamma(1-k_{f}) \sum_{m \geq 1} \sum_{n-m \geq 1} \alpha(m) \beta(n)\left(n^{k_{f}-1}\Pc_{\kappa-2}^{(1-k_{f},1-\kappa)}\left(1-2\frac{m}{n}\right)-m^{k_{f}-1}\right)q^{n-m},
\end{align*}
as desired.
\end{proof}

\begin{proof}[Second proof of Proposition \ref{Prop:pimixed}]
One copies the first proof until the application of the Lipschitz summation formula, which produced the expression
\begin{align*}
&\pi_{\kappa}\left(fg\right)(\tau) = \frac{(\kappa-1) (2i)^{\kappa}}{4\pi}\frac{(-2 \pi i)^{\kappa}}{(\kappa-1)!} \sum_{m \geq 1} \sum_{n-m \geq 1} \alpha(m)m^{k_f-1} \beta(n) (n-m)^{\kappa-1} \\
& \hspace{6em} \times \int_{0}^{\infty} \Gamma(1-k_f,4\pi m y)y^{\kappa-2} \mathrm{e}^{-4\pi (n-m) y} dy \ q^{n-m}.
\end{align*}

Then, one shows the following two identities. The first relies on the fact that $\kappa = k_f+k_g$ is an integer.
\begin{enumerate}[label=(\roman*), wide, labelwidth=!, labelindent=0pt]
\item Define the homogeneous polynomial 
\begin{align*}
P_{a,b}(X,Y) \coloneqq \sum_{j=0}^{a-2} {j+b-2 \choose j}X^j(X+Y)^{a-j-2} \in \C[X,Y]
\end{align*}
of degree $a-2$. Then we have
\begin{multline*}
\int_{0}^{\infty} \Gamma(1-k_f, 4\pi m y)y^{\kappa-2}\mathrm{e}^{-4\pi r y} dy \\
= -(4\pi)^{1-\kappa}m^{1-k_f}\frac{\Gamma(1-k_f)(\kappa-2)!}{r^{\kappa-1}}\left((r+m)^{1-k_g}P_{\kappa,2-k_f}(r,m)-m^{k_f-1}\right).
\end{multline*}
A proof can be found in \cite[Lemma 4.7]{mertens2016}.

\item If $b \neq 1,2$ then the polynomial $P_{a,b}(X,Y)$ from the first item satisfies
\begin{align*}
P_{a,b}(X,Y) = \sum_{j=0}^{a-2} {a+b-3 \choose a-2-j} {j+b-2 \choose j} (X,Y)^{a-2-j}(-Y)^{j}.
\end{align*}
A proof can be found in \cite[Lemma 5.1]{mertens2016}.
\end{enumerate}

Next, one proceeds by writing
\begin{align*}
\pi_{\kappa}\left(fg\right)(\tau) = -\Gamma(1-k_f) \sum_{m \geq 1} \sum_{n-m \geq 1} \alpha(m) \beta(n)\left(n^{1-k_g}P_{\kappa,2-k_f}(n-m,m)-m^{k_f-1}\right)q^{n-m},
\end{align*}
according to the first identity, and then writing
\begin{align*}
& n^{1-k_g}P_{\kappa,2-k_f}(n-m,m) = \sum_{j=0}^{\kappa-2} {k_{g}-1 \choose \kappa-2-j} {j-k_f \choose j} n^{k_{f}-1-j}(-m)^{j} \\
&= \frac{\Gamma(k_g)n^{k_{f}-1}}{(\kappa-2)! \ \Gamma(1-k_f)} \sum_{j=0}^{\kappa-2} {\kappa-2 \choose j} \frac{1}{j+1-k_f}\left(-\frac{m}{n}\right)^j = n^{k_f-1}\Pc_{\kappa-2}^{(1-k_f,1-\kappa)}\left(1-2\frac{m}{n}\right),
\end{align*}
by virtue of the second identity. Summing up, one obtains
\begin{align*}
\pi_{\kappa}\left({f}g\right)(\tau) = -\Gamma(1-k_{f}) \sum_{m \geq 1} \sum_{n-m \geq 1} \alpha(m) \beta(n)\left(n^{k_{f}-1}\Pc_{\kappa-2}^{(1-k_{f},1-\kappa)}\left(1-2\frac{m}{n}\right)-m^{k_{f}-1}\right)q^{n-m},
\end{align*}
as claimed.
\end{proof}

\section{Proof of Theorem \ref{thm:main} and Theorem \ref{thm:varmain}} \label{sec:mainproofs}
We collect the results needed to prove Theorem \ref{thm:main} and Theorem \ref{thm:varmain}. We begin by rewriting the definitions of $\Fcm$ and $\Gcm$.

\begin{lemma} \label{lem:fmrew}
We have
\begin{align*}
\Fcm(\tau) &= \frac{2}{\Gamma\left(-\frac{1}{2}\right)}\sum_{m \geq 1}\chi(m) m\Gamma\left(-\frac{1}{2},4\pi m^2 v\right)q^{-m^2},
\end{align*}
and
\begin{align*}
\Gcm(\tau) &= \frac{2}{\Gamma\left(-\frac{1}{2}\right)}\sum_{m \geq 1} m\Gamma\left(-\frac{1}{2},4\pi m^2 v\right)q^{-m^2} - \frac{1}{2\pi v^{\frac{1}{2}}}.
\end{align*}
\end{lemma}

\begin{proof}
We compute
\begin{align*}
& - \left(2\pi\right)^{-\frac{1}{2}}i \int_{-\overline{\tau}}^{i\infty}\frac{\tc\left(w\right)}{\left(-i\left(w+\tau\right)\right)^{\frac{3}{2}}} dw = - \left(2\pi\right)^{-\frac{1}{2}}i\sum_{m \geq 1} \chi(m) \int_{2iv}^{i\infty} \frac{\mathrm{e}^{2\pi i m^2(z-\tau)}}{(-iz)^{\frac{3}{2}}} dz \\
& = \left(2\pi\right)^{-\frac{1}{2}}\sum_{m \geq 1} \chi(m)\left(\int_{2v}^{\infty} x^{-\frac{3}{2}}\mathrm{e}^{-2\pi m^2x}dx\right) q^{-m^2} = \sum_{m \geq 1} \chi(m)m \left(\int_{4\pi m^2v}^{\infty} t^{-\frac{1}{2}-1}\mathrm{e}^{-t}dt\right) q^{-m^2},
\end{align*}
and the first claim follows directly, since $\frac{2}{\Gamma\left(-\frac{1}{2}\right)} = -\frac{1}{\sqrt{\pi}}$. To prove the second claim, it remains to separate the constant term of $\ti$, and next to calculate
\begin{align*}
\frac{i}{\pi\sqrt{2}} \int_{-\overline{\tau}}^{i\infty}\frac{\frac{1}{2}}{\left(-i\left(w+\tau\right)\right)^{\frac{3}{2}}} dw = -\frac{1}{2\pi\sqrt{2}} \int_{2v}^{\infty} t^{-\frac{3}{2}} dt = -\frac{1}{2\pi v^{\frac{1}{2}}},
\end{align*}
as asserted.
\end{proof}

In addition, we have the following immediate corollary of Proposition \ref{Prop:pimixed}.
\begin{cor} \label{Cor:pifinal}
Let $f(\tau) \coloneqq f^+(\tau)+f^{-}(\tau)$ be the splitting of $f$ into its holomorphic and non-holomorphic part. Assume the notation and hypotheses as in Proposition \ref{Prop:pimixed}. Then
\begin{multline*}
\pi_{\kappa}\left(fg\right)(\tau) \\
= \left(f^{+}g\right)(\tau) - \Gamma(1-k_{f}) \sum_{m \geq 1} \sum_{n-m \geq 1} \alpha(m) \beta(n)\left(n^{k_{f}-1}\Pc_{\kappa-2}^{(1-k_{f},1-\kappa)}\left(1-2\frac{m}{n}\right)-m^{k_{f}-1}\right)q^{n-m}.
\end{multline*}
\end{cor}

Moreover, we combine the properties of the slash-operator, the shadow operator, and of the holomorphic projection operator. This yields a third preparatory result.
\begin{prop}[\protect{\cite[Proposition 2.3]{mertens2019mock}}] \label{Prop:Mod}
Let $f \colon \H \to \C$ be a translation invariant function such that $\vert f(\tau) \vert v^{\delta}$ is bounded on $\H$ for some $\delta > 0$. If the weight $k$ holomorphic projection of $f$ vanishes identically for some $k > \delta + 1$ and $\xi_k f$ is modular of weight $2-k$ for some subgroup $\Gamma < \slz$, then $f$ is modular of weight $k$ for $\Gamma$.
\end{prop}

\begin{proof}
This is a straightforward adaption of \cite[Proposition 3.5]{brikazwe}. Let $\gamma \in \Gamma$. Then the modularity of $\xi_k f$ implies that
\begin{align*}
\xi_k\left(\left(f\vert_k\gamma\right) -f \right) = \left(\xi_k f\right)\vert_{2-k}\gamma- \xi_k f = 0.
\end{align*}
Hence, $\left(f\vert_k\gamma\right) -f$ is holomorphic. This yields
\begin{align*}
\left(f\vert_k\gamma\right) -f = \pi_{\kappa}\left(\left(f\vert_k\gamma\right) -f\right) = \pi_{\kappa}\left(f\right)\vert_k\gamma - \pi_{\kappa}\left(f\right),
\end{align*}
and by assumption the right hand side vanishes. This proves the claim.
\end{proof}

\begin{rmk}
The subtle growth conditions are required to include the case $\pi_2$, and are clearly satisfied if we deal with higher weight holomorphic projections, in which case the integral defining $\pi_k$ converges absolutely.
\end{rmk}

\begin{proof}[Proof of Theorem \ref{thm:main}]
We need to check the three conditions required by the definition of a harmonic Maa{\ss} form.
\begin{enumerate}[label=(\alph*), wide, labelwidth=!, labelindent=0pt]
\item Growth conditions. 
Recall that $\tp$ is a cusp form, namely it decays exponentially towards all cusps. In turn, the function $\Fcp$ admits at most linear exponential growth towards all cusps. Note that in particular $i\infty$ is a removable singularity of $\Fcp$, since $i\infty$ is a simple zero of both $\tp$ and $\tp\Fcp$. To inspect the non-holomorphic part, we have
\begin{align*}
\Gamma\left(-\frac{1}{2}, 4\pi m^2y\right)q^{-m^2} \sim \left(4 \pi m^2 v\right)^{-\frac{3}{2}}\mathrm{e}^{-2\pi m^2 v}, \quad v \to \infty,
\end{align*}
and hence the function $\Fcm$ decays exponentially towards the cusp $i\infty$. By the transformation properties of $\tc$ under the full modular group $\slz$, we deduce that $\Fcm$ is of moderate growth towards all cusps. This establishes the growth condition required by Definition \ref{def:maassforms}.

As pointed out in the introduction after Corollary \ref{cor:hurwitz}, the Fourier coefficients of $\tp\Fcp$ at $i\infty$ are of moderate growth, wherefore the growth of the function $\tp\Fcp$ towards any cusp is moderate. One may see this by choosing suitable scaling matrices, whose action yields additional polynomial factors in $\tau$. Consequently, the growth of $\tp\Fc$ is moderate. This justifies the existence of $\pi_{3}\left(\tp\Fc\right)$ as well as the application of Proposition \ref{Prop:Mod} to $\tp\Fcp$ during the upcoming item.

\item The transformation law.
First, we compute
\begin{align*}
n^{\frac{1}{2}}\Pc_1^{(-\frac{1}{2},-2)}\left(1-2\frac{m}{n}\right) -m^{\frac{1}{2}} = n^{\frac{1}{2}}\left(\frac{1}{2}+\frac{m}{2n}\right)-m^{\frac{1}{2}} = \frac{\left(m^{\frac{1}{2}}-n^{\frac{1}{2}}\right)^2}{2n^{\frac{1}{2}}}.
\end{align*}

Comparing our initial setting with Proposition \ref{Prop:pimixed} and Lemma \ref{lem:shadowmixed}, we need to switch to squares above. By virtue of Lemma \ref{lem:fmrew} and the definition of $\tp$, we have the coefficients
\begin{align*}
\alpha\left(m^2\right) = \frac{2}{\Gamma\left(-\frac{1}{2}\right)}\chi(m), \qquad \beta\left(n^2\right) = \psi(n)n.
\end{align*}

We rewrite the generating function of $\sdf$ as
\begin{align*}
\left(\Fcp\tp\right)(\tau) &= \sum_{n \geq 1} \sdf(n)q^n = \sum_{n \geq 1}\sum_{d \in D_n} \chi\left(\frac{\frac{n}{d}-d}{2}\right)\psi\left(\frac{\frac{n}{d}+d}{2}\right) d^2 q^n \\
&= \sum_{d \geq 1} \sum_{\substack{j \geq d \\ j \equiv d \pmod*{2}}} \chi\left(\frac{j-d}{2}\right)\psi\left(\frac{j+d}{2}\right) d^2 q^{dj} \\
&= \sum_{m \geq 1} \sum_{n-m \geq 1} \chi(m)\psi(n)\left(n-m\right)^2q^{n^2-m^2},
\end{align*}
and apply Corollary \ref{Cor:pifinal} to obtain
\begin{align*}
\pi_{3}\left(\Fc\tp\right)(\tau) = \left(\Fcp\tp\right)(\tau) - \sum_{m \geq 1} \sum_{n-m \geq 1} \chi(m)\psi(n)\left(m-n\right)^2q^{n^2-m^2} = 0.
\end{align*}

Furthermore, we apply Lemma \ref{lem:shadowmixed}, obtaining
\begin{align*}
\xi_{3} \left(\Fc\tp\right)(\tau) = -(4\pi)^{-\frac{1}{2}}\frac{2}{\Gamma\left(-\frac{1}{2}\right)}v^{\frac{3}{2}} \sum_{m \geq 1} \overline{\chi(m)}q^{m^2} \sum_{n \geq 1}\overline{\psi(n)nq^{n^2}} = \frac{1}{2\pi}v^{\frac{3}{2}} \theta_{\overline{\chi}}(\tau)\frac{\left\vert\tp(\tau)\right\vert^2}{\tp(\tau)}.
\end{align*}

We observe that $\xi_{3} \left(\Fc\tp\right)(\tau)$ is modular of weight $-1$ for $\Gamma_0(4\Mc^2) \cap \Gamma_0(4\Mp^2) = \Gamma_0\left(\lcm\left(4\Mc^2,4\Mp^2\right)\right)$ with Nebentypus $\overline{\chi} \cdot \psi^{-1} \cdot \chi_{-4}^{-1}$. Indeed, for any $\gamma = \left(\begin{smallmatrix} a & b \\ c & d \end{smallmatrix}\right) \in \Gamma_0\left(\lcm\left(4\Mc^2,4\Mp^2\right)\right)$ we have
\begin{align*}
\xi_{3} \left(\Fc\tp\right)(\gamma\tau) &= \frac{1}{2\pi}\frac{v^{\frac{3}{2}}}{\left\vert c\tau+d\right\vert^3} \overline{\chi(d)}(c\tau+d)^{\frac{1}{2}}\theta_{\overline{\chi}}(\tau) \frac{\left\vert\psi(d)\chi_{-4}(d)(c\tau+d)^{\frac{3}{2}}\tp(\tau)\right\vert^2}{\psi(d)\chi_{-4}(d)(c\tau+d)^{\frac{3}{2}}\tp(\tau)} \\
&= \overline{\chi(d)}\psi(d)^{-1}\chi_{-4}^{-1}(d)(c\tau+d)^{-1} \xi_{3} \left(\Fc\tp\right)(\tau).
\end{align*}
Finally, Proposition \ref{Prop:Mod} applies directly, because the growth conditions are met thanks to absolute convergence. We deduce that $\Fc\tp$ is modular of weight $3$ with respect to the same data, as desired.

\item Harmonicity.
Clearly, $\Fcp$ is holomorphic away from the zeros of $\tp$. Hence, the Cauchy-Riemann equations imply $\Delta_{\frac{3}{2}}\Fcp = 0$ directly. The computation for $\Fcm$ is standard, so we only sketch its results. It holds that
\begin{align*}
& \left(\frac{\partial}{\partial u} + i\frac{\partial}{\partial v}\right)\Gamma\left(-\frac{1}{2},4\pi m^2 v\right)q^{-m^2} = -\frac{i}{2}\frac{\mathrm{e}^{-2\pi m^2(iu+v)}}{\pi^{\frac{1}{2}}mv^{\frac{3}{2}}} , \\
& \left(\frac{\partial^2}{\partial u^2}+\frac{\partial^2}{\partial v^2}\right)\Gamma\left(-\frac{1}{2},4\pi m^2 v\right)q^{-m^2} = \frac{3}{4}\frac{\mathrm{e}^{-2\pi m^2(iu+v)}}{\pi^{\frac{1}{2}}mv^{\frac{5}{2}}}, \\
& \left(\Delta_{\frac{3}{2}}\Fcm\right)(\tau) = -v^2\left(\frac{3}{4}\frac{\mathrm{e}^{-2\pi m^2(iu+v)}}{\pi^{\frac{1}{2}}mv^{\frac{5}{2}}}\right) + \frac{3i}{2}v\left(-\frac{i}{2}\frac{\mathrm{e}^{-2\pi m^2(iu+v)}}{\pi^{\frac{1}{2}}mv^{\frac{3}{2}}}\right) = 0,
\end{align*}
and thus $\Delta_{\frac{3}{2}} \Fc = 0$ away from the zeros of $\tp$.
\end{enumerate}
Altogether, this completes the proof, since the shadow is a byproduct of the second item.
\end{proof}

\begin{proof}[Proof of Theorem \ref{thm:varmain}]
The proof of Theorem \ref{thm:varmain} uses the same ideas as the proof of Theorem \ref{thm:main}, so we just emphasize the differences. Recall from Lemma \ref{lem:fmrew} that
\begin{align*}
\Gcm(\tau) = \frac{2}{\Gamma\left(-\frac{1}{2}\right)}\sum_{m \geq 1} m\Gamma\left(-\frac{1}{2},4\pi m^2 v\right)q^{-m^2} - \frac{1}{2\pi v^{\frac{1}{2}}}.
\end{align*}
Therefore, to compute $\pi_{3}(\tp\Gcm)$, it suffices to deal with the second term. We see that 
\begin{align*}
-\frac{\tp(\tau)}{2\pi v^{\frac{1}{2}}}
\end{align*}
is translation invariant, vanishes at $i\infty$, and has a removable singularity at all other cusps inspecting the order of vanishing as $v \searrow 0$. Hence the integral defining its weight $3$ holomorphic projection exists and converges absolutely. The computation begins exactly as in the proof of Proposition \ref{Prop:pimixed} and we employ the Lipschitz summation formula. This yields
\begin{align*}
\pi_{3}\left(-\frac{\tp(\tau)}{2\pi v^{\frac{1}{2}}}\right) &= -\frac{2(2i)^3}{8\pi^2} \sum_{n \geq 1} \psi(n)n\sum_{j \in \Z} \int_{0}^{\infty} \int_{0}^{1} \frac{y^{\frac{1}{2}}\mathrm{e}^{2\pi i n^2(x+iy)}}{(\tau-x+iy+j)^3}dxdy \\
&= -8\pi \sum_{n \geq 1} \psi(n)n\sum_{j \geq 1} j^2 \int_{0}^{\infty}\int_{0}^{1}  y^{\frac{1}{2}}\mathrm{e}^{2\pi i \left(n^2(x+iy) + j(\tau-x+iy)\right)} dxdy \\
& = -8\pi \sum_{n \geq 1} \psi(n) n^5 \left(\int_{0}^{\infty} y^{\frac{1}{2}}\mathrm{e}^{-4\pi n^2y} dy\right)q^{n^2} = -\frac{1}{2}\sum_{n \geq 1} \psi(n)n^2q^{n^2},
\end{align*}
which also appeared in \cite[Lemma 4.4]{mertens2016} in the framework of mixed harmonic Maa{\ss} forms. Furthermore, it follows by Corollary \ref{Cor:pifinal} (with identical parameters of the Jacobi polynomial as in the proof of Theorem \ref{thm:main}) that
\begin{align*}
\pi_{3}\left(\Gc\tp\right)(\tau) = \left(\Gcp\tp\right)(\tau) - \sum_{m \geq 1} \sum_{n-m \geq 1} \psi(n)\left(m-n\right)^2q^{n^2-m^2} -\frac{1}{2}\sum_{n \geq 1} \psi(n)n^2q^{n^2}.
\end{align*}
In addition, we note that
\begin{align*}
\sum_{n \geq 1} \sdft(n) q^n &= \sum_{d \geq 1} \sum_{\substack{j \geq d \\ j \equiv d \pmod*{2}}} \psi\left(\frac{j+d}{2}\right)d^2q^{jd} = \sum_{m \geq 1} \sum_{n \geq m+1} \psi(n)(n-m)^{2}q^{n^2-m^2},
\end{align*}
where we substituted $d = n-m$, $j = n+m$ in the last equation. Collecting these observations and inserting the definition of $\Gcp$, we obtain
\begin{align*}
\pi_{3}\left(\Gc\tp\right)(\tau) = 0.
\end{align*}
Moreover,
\begin{align*}
\xi_{3} \left(\Gc\tp\right)(\tau) = \frac{1}{2\pi} v^{\frac{3}{2}}\overline{\tp(\tau)}\sum_{m \geq 1} q^{m^2} + \frac{1}{4\pi}v^{\frac{3}{2}}\overline{\tp(\tau)} = \frac{1}{4\pi}v^{\frac{3}{2}}\frac{\left\vert\tp(\tau)\right\vert^2}{\tp(\tau)} \sum_{m \in \Z} q^{m^2},
\end{align*}
which is modular of weight $-1$ on $\Gamma_0\left(4\Mp^2\right) \cap \Gamma_0(4) = \Gamma_0\left(\lcm\left(4,4\Mp^2\right)\right) = \Gamma_0(4\Mp^2)$ with Nebentypus $\left(\psi \cdot \chi_{-4}\right)^{-1}$ by the same argument as in the proof of Theorem \ref{thm:varmain}. This establishes weight $3$ modularity of $\Gc\tp$ via Proposition \ref{Prop:Mod} again. 
\newline
Additionally, we clearly have
\begin{align*}
\Delta_{\frac{3}{2}}\left(-\frac{1}{2\pi v^{\frac{1}{2}}}\right) = 0,
\end{align*}
and hence harmonicity is preserved. Finally, the growth properties of $\Gcp$ towards all cusps agree verbatim with $\Fcp$. Summing up, this establishes the Theorem.
\end{proof}

\section{Proof of Theorem \ref{thm:padicthm}} \label{sec:remthms}

\begin{proof}[Proof of Theorem \ref{thm:padicthm}]
We prove the first claim. On one hand, by the same computation as during the proof of Theorem \ref{thm:main} we infer
\begin{align*}
\pi_{3}\left(\tp\left(p^{2a}\tau\right)\Fc(\tau)\right) = \tp\left(p^{2a}\tau\right)\Fcp(\tau) + \sum_{r \geq 1} \left(\sum_{\substack{m,n \geq 1 \\ (p^a n)^2-m^2=r}} \chi(m)\psi(n)\left(m - p^a n\right)^2 \right)q^{r} \eqqcolon g(\tau).
\end{align*}
Invoking Theorem \ref{thm:main} and the properties of holomorphic projection we deduce that $g$ is a modular form of weight $3$ on $\Gamma_0\left(\lcm\left(4\Mp^2 p^{2a},4\Mc^2\right)\right)$ with Nebentypus $\overline{\chi}$. However, clearly $-I \in \Gamma_0(N)$ for every level $N$ and hence $g$ vanishes identically (recall that $\chi$ is assumed to be even throughout).
\newline
On the other hand, the inner sum can be rewritten as a sum over small divisors of $r$. To this end, the set of admissible small divisors is given by
\begin{align*}
D_r(p) \coloneqq \left\{ d\mid r \ \colon \ 1 \leq d \leq \frac{r}{d}, \quad d \equiv \frac{r}{d} \ \pmod*{2} , \quad d+\frac{r}{d} \equiv 0 \ \pmod*{2p^a} \right\},
\end{align*}
and exactly as in the proof of Theorem \ref{thm:main} we see that
\begin{align*}
\sum_{\substack{m,n \geq 1 \\ (p^a n)^2-m^2=r}} \chi(m)\psi(n)\left(m - p^a n\right)^2 = \sum_{d \in D_r(p)} \chi\left(\frac{\frac{n}{d}-d}{2}\right)\psi\left(\frac{\frac{n}{d}+d}{2p^a}\right) d^2.
\end{align*}
If we apply the operator $U(p^b)$ to $g$ then we need to replace $r$ by $p^b r$ everywhere above. This produces the condition
\begin{align*}
d+\frac{p^b r}{d} \equiv 0 \pmod*{2p^a}
\end{align*}
in the set of admissible small divisors, which eventually forces
\begin{align*}
d \equiv 0 \ \pmod*{p^{\textnormal{min}(a,b)}},
\end{align*}
since $d$ is a divisor of $r$. Combining, we arrive at
\begin{align*}
0 = g(\tau) = g(\tau) \Big\vert \ U(p^b) \equiv \left(\tp\left(p^{2a}\tau\right)\Fcp(\tau)\right) \Big\vert \ U(p^b) \ \pmod*{p^{\textnormal{min}(a,b)}},
\end{align*}
as claimed.
\newline
The proof of the second claim is completely analogous. The character $\chi$ is trivial, $\Mc = 1$, and one can remove the condition $d \equiv \frac{r}{d} \ \pmod*{2}$ from the definition of the set of admissible small divisors. However, this does not affect the rest of the proof and we provided the necessary computations during the proof of Theorem \ref{thm:varmain} essentially.
\end{proof}

\section{Proof of Proposition \ref{Prop:parthetalerch} and of Proposition \ref{Prop:alprop}} \label{sec:alparproof}

\begin{proof}[Proof of Proposition \ref{Prop:alprop}:]
The first step is to apply geometric expansion. We compute
\begin{align*}
\sum_{n \geq 1} \sdf(n) q^n &= \sum_{m \geq 1} \sum_{n-m \geq 1} \chi(m)\psi(n)\left(n-m\right)^2q^{n^2-m^2} = \sum_{m \geq 1} \sum_{s \geq 1} \chi(m)\psi(m+s) s^2 q^{s^2+2ms} \\
&= \sum_{s \geq 1} \sum_{a \geq 0} \sum_{b=0}^{\Mc-1} \chi(b)\psi(s+b)s^2q^{s^2+2(a\Mc+b)s} - \chi(0)\psi(s+b)s^2q^{s^2} \\
&= \sum_{b=1}^{\Mc-1} \chi(b) \sum_{s \geq 1} \psi(s+b)s^2 \frac{q^{s^2+2bs}}{1-q^{2\Mc s}},
\end{align*}
where we have used the assumption $\Mp \mid \Mc$ after the substitution $m = a\Mc+b$ and the assumption that $\chi$ is non-trivial in the last equation.
The second step is to convert the sum in $s$ to a sum over $\Z$ instead of $\N$. Note that
\begin{align*}
2 \sum_{s \geq 1} \psi(s+b) s^2\frac{q^{s^2+2bs}}{1-q^{2\Mc s}} &= \sum_{s \geq 1} \psi(s+b)s^2 \frac{q^{s^2+2bs}}{1-q^{2\Mc s}} + \sum_{s \leq -1} \psi(-s+b)s^2 \frac{q^{s^2-2bs}}{1-q^{-2\Mc s}} \\
&= \sum_{s \geq 1} \psi(s+b) s^2\frac{q^{s^2+2bs}}{1-q^{2\Mc s}} + \sum_{s \leq -1} \psi(s-b) s^2\frac{q^{s^2+2\Mc s-2bs}}{1-q^{2\Mc s}},
\end{align*}
using that $\psi$ is odd. The key observation is
\begin{align*}
\sum_{b=1}^{\Mc-1} \chi(b) \sum_{s \leq -1} \psi(s-b)s^2 \frac{q^{s^2+2\Mc s-2bs}}{1-q^{2\Mc s}} &= \sum_{b=1}^{\Mc-1} \chi(\Mc-b) \sum_{s \leq -1} \psi(s-(\Mc-b)) s^2\frac{q^{s^2+2bs}}{1-q^{2\Mc s}} \\
&= \sum_{b=1}^{\Mc-1} \chi(b) \sum_{s \leq -1} \psi(s+b)s^2 \frac{q^{s^2+2bs}}{1-q^{2\Mc s}}.
\end{align*}
using that $\chi$ is even. Thus, we have
\begin{align*}
\sum_{n \geq 1} \sdf(n) q^n = \frac{1}{2} \sum_{b=1}^{\Mc-1} \chi(b) \sum_{s \in \Z} \psi(s+b)s^2 \frac{q^{s^2+2bs}}{1-q^{2\Mc s}},
\end{align*}
since the constant term in $s$ vanishes as well. The third step is to substitute $s = n\Mc+c$ to isolate $\psi$ from its $s$-dependency (recall $\Mp \mid \Mc$), getting
\begin{align*}
& \sum_{n \geq 1} \sdf(n) q^n = \frac{1}{2} \sum_{b=1}^{\Mc-1} \chi(b)\sum_{c=0}^{\Mc-1}\psi(b+c) \sum_{n \in \Z} (n\Mc+c)^2 \frac{q^{(n\Mc+c)^2+2b(n\Mc+c)}}{1-q^{2\Mc(n\Mc+c)}},
\end{align*}
from which we read off the claim.
\end{proof}

\begin{proof}[Proof of Proposition \ref{Prop:parthetalerch} (i):]
We provide a purely computational proof. To this end, let
\begin{align*}
J(\tau) \coloneqq \int_{-\frac{1}{2}}^{\frac{1}{2}} \mathrm{e}^{2\pi i (2t)} A_2\left(t-\frac{\tau}{2},0;\tau\right) dt.
\end{align*}
We split the sum defining $A_2$ into its positive, constant, and negative summands (with respect to the summation variable $n$). Next, we note that $\vert \mathrm{e}^{\pm 2\pi i t} q^{n \mp \frac{1}{2}}\vert < 1$ for every $n \geq 1$ and $t \in \R$, which enables us to apply geometric expansion, which yields
\begin{align*}
\sum_{n \geq 1} \frac{q^{n(n+1)}}{1-\mathrm{e}^{2\pi i \left(t-\frac{\tau}{2}\right)}q^n} &= \sum_{n \geq 1} \sum_{j \geq 0} q^{n(n+1)+jn}\mathrm{e}^{2\pi i j\left(t-\frac{\tau}{2}\right)}, \\
\sum_{n \leq -1} \frac{q^{n(n+1)}}{1-\mathrm{e}^{2\pi i \left(t-\frac{\tau}{2}\right)}q^n} &= -\sum_{n \geq 1} \frac{q^{n^2}\mathrm{e}^{-2\pi i \left(t-\frac{\tau}{2}\right)}}{1-\mathrm{e}^{-2\pi i \left(t-\frac{\tau}{2}\right)}q^{n}} = -\sum_{n \geq 1} \sum_{j \geq 0} q^{n^2+jn}\mathrm{e}^{-2\pi i (1+j) \left(t-\frac{\tau}{2}\right)}.
\end{align*}

We collect all terms, getting
\begin{align*}
J(\tau) =& \sum_{n \geq 1} \sum_{j \geq 0} \int_{-\frac{1}{2}}^{\frac{1}{2}} \mathrm{e}^{2\pi i (2t)} \mathrm{e}^{2\pi i \left(t-\frac{\tau}{2}\right)} q^{n(n+1)+jn}\mathrm{e}^{2\pi i j\left(t-\frac{\tau}{2}\right)}dt + \int_{-\frac{1}{2}}^{\frac{1}{2}} \mathrm{e}^{2\pi i (2t)} \mathrm{e}^{2\pi i \left(t-\frac{\tau}{2}\right)} \frac{1}{1-\mathrm{e}^{2\pi i \left(t-\frac{\tau}{2}\right)}}dt \\
&- \sum_{n \geq 1} \sum_{j \geq 0} \int_{-\frac{1}{2}}^{\frac{1}{2}} \mathrm{e}^{2\pi i (2t)} \mathrm{e}^{2\pi i \left(t-\frac{\tau}{2}\right)} q^{n^2+jn}\mathrm{e}^{-2\pi i (1+j) \left(t-\frac{\tau}{2}\right)}dt.
\end{align*}
The first term vanishes and the third term reduces to the case $j = 2$. Moreover, the second term simplifies to
\begin{align*}
\int_{-\frac{1}{2}}^{\frac{1}{2}} \frac{\mathrm{e}^{2\pi i (2t)} \mathrm{e}^{2\pi i \left(t-\frac{\tau}{2}\right)}}{1-\mathrm{e}^{2\pi i \left(t-\frac{\tau}{2}\right)}}dt = \int_{-\frac{1}{2}}^{\frac{1}{2}} -\mathrm{e}^{2\pi i t+\pi i \tau} - \frac{\mathrm{e}^{3\pi i \tau}}{\mathrm{e}^{2\pi i t}-\mathrm{e}^{\pi i \tau}} - q - \mathrm{e}^{2\pi i (2t)}dt = -q,
\end{align*}
and we arrive at
\begin{align*}
J(\tau) = -q - \sum_{n \geq 1} q^{n^2+2n+1} = - \sum_{n \geq 1} q^{n^2},
\end{align*}
as claimed.
\end{proof}

\begin{proof}[Proof of Proposition \ref{Prop:parthetalerch} (ii):]
We follow the original proof. Let $z^* = -\frac{\tau}{2} - \frac{1}{2}$ and $P_{z^*}$ be the fundamental parallelogram around $0$. Then $P_{z^*}$ contains no other zeros of $\vartheta(z;\tau)^2$. In the notation of \cite{brirozwe16} we set
\begin{align*}
\vartheta_{-2,0,1}^+(z;\tau) \coloneqq \sum_{n \geq 1} q^{n^2}\zeta^{2n},
\end{align*}
and the idea is to evaluate the integral
\begin{align*}
I(\tau) \coloneqq \int_{\partial P_{z^*}} \frac{\vartheta_{-2,0,1}^+(z;\tau)}{\vartheta(z;\tau)^2}dz
\end{align*}
in two different ways. On one hand, the Laurent expansion of $\frac{1}{\vartheta(z;\tau)^2}$ with respect to $z$ around $0$ has the shape
\begin{align*}
\frac{1}{\vartheta(z;\tau)^2} = \frac{f_{-2}(\tau)}{(2\pi i z)^2} + \frac{f_{-1}(\tau)}{2\pi i z} + O(1),
\end{align*}
and thus
\begin{align*}
I(\tau) &= 2\pi i \text{Res}\left(0, \frac{\vartheta_{-2,0,1}^+(\ \cdot \ ;\tau)}{\vartheta(\ \cdot \ ;\tau)^2}\right) = 2\pi i \frac{d}{dz}\left(z^2\frac{\vartheta_{-2,0,1}^+(z;\tau)}{\vartheta(z;\tau)^2}\right)\Bigg\vert_{z=0} \\
&= f_{-1}(\tau)\vartheta_{-2,0,1}^+(0;\tau) + f_{-2}(\tau)\left(\frac{1}{2\pi i}\frac{d}{dz}\vartheta_{-2,0,1}^+(z;\tau)\right)\Bigg\vert_{z=0} = f_{-1}(\tau) \sum_{n \geq 1} q^{n^2} + 2f_{-2}(\tau)\sum_{n \geq 1} nq^{n^2}.
\end{align*}

On the other hand, the integrand is one-periodic in $z$, and hence the integrals along the vertical edges cancel each other. Moreover, we have the elliptic transformation property
\begin{align*}
\vartheta(z+\lambda\tau+\mu;\tau) = (-1)^{\lambda+\mu}\mathrm{e}^{-\pi i\left(\lambda^2\tau+2\lambda z\right)}\vartheta(z;\tau)
\end{align*}
for any $\lambda, \mu \in \Z$. Combining yields
\begin{align*}
I(\tau) &= \int_{z^*}^{z^*+1} \frac{\vartheta_{-2,0,1}^+(z;\tau)}{\vartheta(z;\tau)^2}dz - \int_{z^*+\tau}^{z^*+\tau+1} \frac{\vartheta_{-2,0,1}^+(z;\tau)}{\vartheta(z;\tau)^2}dz \\
&= \int_{z^*}^{z^*+1} \left(\frac{\vartheta_{-2,0,1}^+(z;\tau)}{\vartheta(z;\tau)^2} - \frac{\vartheta_{-2,0,1}^+(z+\tau;\tau)}{\vartheta(z+\tau;\tau)^2} \right) dz \\
&= \int_{z^*}^{z^*+1} \left(\frac{\vartheta_{-2,0,1}^+(z;\tau)}{\vartheta(z;\tau)^2} - \mathrm{e}^{2\pi i (\tau+2z)}\frac{\vartheta_{-2,0,1}^+(z+\tau;\tau)}{\vartheta(z;\tau)^2} \right) dz.
\end{align*}
The last step is to compute
\begin{align*}
\vartheta_{-2,0,1}^+(z;\tau) - q\zeta^2\vartheta_{-2,0,1}^+(z+\tau;\tau) &= \sum_{n \geq 1} q^{n^2}\zeta^{2n} - \sum_{n \geq 1} q^{(n+1)^2}\zeta^{2(n+1)} = q\zeta^2.
\end{align*}
Hence,
\begin{align*}
I(\tau) = q\int_{z^*}^{z^*+1} \frac{\mathrm{e}^{2\pi i (2z)}}{\vartheta(z;\tau)^2}dz = \int_{-\frac{1}{2}}^{\frac{1}{2}} \frac{\mathrm{e}^{2\pi i (2t)}}{\vartheta\left(t-\frac{\tau}{2};\tau\right)^2}dt.
\end{align*}
This completes the proof.
\end{proof}



\end{document}